\documentclass[12pt]{article}
\usepackage{mathrsfs}
\usepackage{fancyhdr,graphicx}
\usepackage{amsfonts}
\usepackage{amssymb}
\usepackage{amsthm}
\usepackage{newlfont}
\usepackage{amsmath}
\usepackage[mathlines]{lineno}
\numberwithin{equation}{section}

\newtheorem{thm}{Theorem}[section]
\newtheorem{lemma}[thm]{Lemma}
\newtheorem{cor}[thm]{Corollary}
\newtheorem{pro}[thm]{Proposition}
\newtheorem{defn}[thm]{Definition}

\usepackage{latexsym, bm}
\title{Further results on the global cyclicity index of graphs}

\author{Yujun Yang$^{1,2}$
\\\small{1. School of Mathematics and Information Science, Yantai University,}
\\\small{Yantai, Shandong, 264005, P.R. China}
\\\small{2. School of Mathematics, Shandong University, Jinan, Shandong, 250010, P.R. China}
\\\small{E-mail address: yangyj@yahoo.com}}

\date{}

\begin{document}
\maketitle \baselineskip 18pt
\begin{abstract}
Being motivated in terms of mathematical concepts from the theory of electrical networks, Klein \& Ivanciuc introduced and studied a new graph-theoretic cyclicity index--the global cyclicity index (Graph cyclicity, excess conductance, and resistance deficit, J. Math. Chem. 30 (2001) 271--287). In this paper, by utilizing techniques from graph theory, electrical network theory and real analysis, we obtain some further results on this new cyclicity measure, including the strictly monotone increasing property, some lower and upper bounds, and some Nordhuas-Gaddum-type results. In particular, we establish a relationship between the global cyclicity index $C(G)$  and the cyclomatic number $\mu(G)$ of a connected graph $G$ with $n$ vertices and $m$ edges:
$$\frac{m}{n-1}\mu(G)\leq C(G)\leq \frac{n}{2}\mu(G).$$
\vspace{0.5cm}

\vspace{0.5cm}

\end{abstract}


\section{Introduction}

There are different possible measures of ``cyclicity" of a graph $G=(V(G),E(G))$. One simple such
traditional fundamental measure is the cyclomatic number $\mu(G)$ (also called the first Betti number, the nullity or
the cycle rank ) which is defined for a connected graph $G$ with $n$ vertices and $m$ edges as
$$\mu(G)=m-n+1.$$
Motivated from electrical network theory, Klein and Ivanciuc proposed a new cyclicity measure. This new cyclicity measure is established on the basis of the novel concept of resistance distance proposed in the excellent paper by Klein and Randic \cite{kr}. As an intrinsic graph metric, the \textit{resistance distance} on a connected graph $G$ is a distance function $\Omega_G: V(G)\times V(G)\rightarrow \mathbb{R}$ which may be defined \cite{kr} for $i,j\in V$ as the net effective resistance between vertices $i$ $\&$ $j$ when unit resistors are identified with each edge of $G$. Comparing to the traditional (shortest path) distance function $d_G:V(G)\times V(G)\rightarrow \mathbb{N}$, it is well known that $\Omega_G(i,j)$ equals the length $d_G(i,j)$ of the shortest path between $i$ and $j$ iff there is a unique single path between $i$ and $j$, while if there is more than one path, then $\Omega_G(i,j)$ is strictly less than $d_G(i,j)$.  Thence the conductance excess
$\sigma_G(i,j)-1/d_G(i,j)$ indicate in some manner the presence of cyclicity in the portion of the graph
interconnecting $i$ and $j$, where $\sigma_G(i,j)=1/\Omega_G(i,j)$ is known as the effective conductance between $i$ and $j$. To measure the cyclicity of a graph $G$, Klein and Ivanciuc \cite{ki} proposed the \textit{global cyclicity index} $C(G)$ as
\begin{equation}
C(G)=\sum_{i\sim j}\big[\sigma_G(i,j)-1/d_G(i,j)\big],
\end{equation}
where $i\sim j$ means $i$ $\&$ $j$ are adjacent and the sum is over all edges of $G$. Since $d_G(i,j)=1$ for $i\sim j$, $C(G)$ can also be written as
\begin{equation}
C(G)=\sum_{i\sim j}\big[\sigma_G(i,j)-1\big]=\sum_{i\sim j}\sigma_G(i,j)-|E(G)|.
\end{equation}

As a new measure of cyclicity of graphs, the global cyclicity index has less degeneracy than the standard cyclomatic number and has some intuitively appealing features. Since the idea of cyclicity is intimately related to measures of connectivity or complexity \cite{tu} and characterization of ``cyclicity" is an aspect of key importance in the study of molecular graphs \cite{bm,bb}, it is worth studying the global cyclicity index from both mathematical and chemical points of view.

In \cite{ki}, Klein and Ivanciuc established a number of theorems for the global cyclicity index of graphs (even not connected). In \cite{yang}, the present author obtained bounds for the global cyclicity index of fullerene graphs. In this paper, we proceed to study this interesting cyclicity measure. In view of the fact that the global cyclicity index of a disconnected graph is the sum of global cyclicity indices of its components \cite[Theorem F]{ki}, we focus ourselves only on connected graphs. By utilizing techniques from graph theory, electrical network theory and real analysis,  we obtain some further results on the global cyclicity number, including the strictly monotone increasing property, some lower and upper bounds, and some Nordhuas-Gaddum-type results. In particular, we establish a relationship between $C(G)$ and the cyclomatic number $\mu(G)$ of a connected graph $G$ with $n$ vertices and $m$ edges:
$$\frac{m}{n-1}\mu(G)\leq C(G)\leq \frac{n}{2}\mu(G).$$

\section{The strictly monotone increasing property}
In this section, we will prove the \textit{strictly monotone increasing property} (SMI-property for short) of the global cycicity index. For this purpose, we first show if a new edge is added to a graph, then the global cyclicity index will be strictly increased. It turns out that the following recursion formula for resistance distances \cite{yk} plays an essential roll in proving this point.

\begin{thm}\label{recur}\cite[Theorem 2.1]{yk}
Let $\Omega$ and $\Omega'$ be resistance distance functions for edge-weighted connected graphs $G$ and $G'$ which are the same except for the
weights $w$ and $w'$ on an edge $e$ with end vertices $i$ and $j$. Then for any $p,q\in V(G)=V(G')$,
\begin{equation}\label{r'}
\Omega'(p,q)=\Omega(p,q)-\frac{\delta\cdot[\Omega(p,i)+\Omega(q,j)-\Omega(p,j)-\Omega(q,i)]^2}{4[1+\delta\cdot\Omega(i,j)]},
\end{equation}
where $\delta\equiv w'-w$.
\end{thm}

Theorem \ref{recur} is applicable to more general edge-weighted graphs and the weight on each edge represents its conductance (the reciprocal of resistance). In this paper, we only consider simple graphs of unit weight on each edge. Especially, it should be pointed out the that if two vertices $i$ and $j$ are not adjacent, then it is equivalent to say that they are connected by an edge of weight 0.

Applying Theorem \ref{recur} to the global cyclicity index, we have the following result.
\begin{thm}\label{diff}
Let $G$ be a connected graph such that $i$ and $j$ are not adjacent in $G$. Denote by $G+ij$ the graph obtained from $G$ by adding a new edge $ij$. Then
\begin{align}\label{dif}
&C(G+ij)-C(G)=\frac{1}{\Omega_G(i,j)}+\notag\\
&\sum_{pq\in E(G)}\frac{[\Omega_G(p,i)+\Omega_G(q,j)-\Omega_G(p,j)-\Omega_G(q,i)]^2}{4\Omega^2_G(p,q)[1+\Omega_G(i,j)]-\Omega_G(p,q)[\Omega_G(p,i)+\Omega_G(q,j)-\Omega_G(p,j)-\Omega_G(q,i)]^2}
\end{align}
\end{thm}
\begin{proof}
\begin{align}\label{e0}
C(G+ij)-C(G)&=\sum\limits_{pq\in E(G+ij)}\sigma_{G+ij}(p,q)-\sum\limits_{pq\in E(G)}\sigma_{G}(p,q)\notag\\
&=\sum\limits_{pq\in E(G+ij)}\big[\frac{1}{\Omega_{G+ij}(p,q)}-1\big]-\sum\limits_{pq\in E(G)}\big[\frac{1}{\Omega_{G}(p,q)}-1\big]\notag\\
&=\frac{1}{\Omega_{G+ij}(i,j)}-1+\sum\limits_{pq\in E(G)}\big[\frac{1}{\Omega_{G+ij}(p,q)}-\frac{1}{\Omega_{G}(p,q)}\big].
\end{align}
Noticing that the weights on edge $e=ij$ in $G+ij$ and $G$ are $1$ and 0, respectively, hence $\delta=w'-w=1-0=1$ and by Theorem \ref{recur}, we have
\begin{equation}\label{e1}
\Omega_{G+ij}(i,j)=\frac{\Omega_{G}(i,j)}{1+\Omega_{G}(i,j)},
\end{equation}
\begin{equation}\label{e2}
\Omega_{G+ij}(p,q)=\Omega_G(p,q)-\frac{[\Omega_G(p,i)+\Omega_G(q,j)-\Omega_G(p,j)-\Omega_G(q,i)]^2}{4[1+\Omega_G(i,j)]}.
\end{equation}
Then Theorem \ref{diff} is obtained by substituting Eqs. (\ref{e1}) and (\ref{e2}) into Eq. (\ref{e0}).
\end{proof}

Furthermore, the difference $C(G+ij)-C(G)$ may be bounded in terms of $\Omega_G(i,j)$ as given in the following proposition. To obtain the proposition, we will use the famous Foster's (first) formula \cite{fo} which demonstrates that the sum of resistance distances between all pairs of adjacent vertices of a connected graph $G$ with $n$ vertices is equal to $n-1$, that is,
\begin{equation}
\sum_{i\sim j}\Omega_G(i,j)=n-1.
\end{equation}
\begin{pro} Let $G$ be a connected graph with $m$ edges. If $i,j\in V(G)$ are not adjacent in $G$, then
\begin{equation}
\frac{1}{\Omega_G(i,j)}<C(G+ij)-C(G)\leq \frac{m+1}{\Omega_G(i,j)},
\end{equation}
with equality if and only if $G$ is the path graph $P_n$ with $i$ and $j$ being its two end vertices.
\end{pro}
\begin{proof}
It is obvious that $C(G+ij)-C(G)\geq \frac{1}{\Omega_G(i,j)}$, with equality if and only if  $\Omega_{G+ij}(p,q)=\Omega_G(p,q)$ holds for each $pq\in E(G)$. However, on one hand, by Eq. (\ref{e2}), we have $\Omega_{G+ij}(p,q)\leq \Omega_G(p,q)$, on the other hand, by Foster's first formula,
\begin{align*}
&\sum_{pq\in E(G)}\Omega_{G+ij}(p,q)=\sum_{pq\in E(G+ij)}\Omega_{G+ij}(p,q)-\Omega_{G+ij}(i,j)=n-1-\Omega_{G+ij}(i,j)\\
<&\sum_{pq\in E(G)}\Omega_{G}(p,q)=n-1.
\end{align*}
Hence we can see that there must exist some $pq\in E(G)$ such that $\Omega_{G+ij}(p,q)< \Omega_G(p,q)$, and consequently,
$$C(G+ij)-C(G)> \frac{1}{\Omega_G(i,j)}.$$

For the upper bound, by the triangle inequality on resistance distances, for any two vertices $p,q\in V(G)$, we have
$|\Omega_G(p,i)-\Omega_G(q,i)|\leq \Omega_G(p,q)$ and $|\Omega_G(p,j)-\Omega_G(q,j)|\leq \Omega_G(p,q)$. Hence
$$[\Omega_G(p,i)+\Omega_G(q,j)-\Omega_G(p,j)-\Omega_G(q,i)]^2\leq 4\Omega^2_G(p,q).$$
Substituting the above inequality into Eq. (\ref{dif}), we have
\begin{align*}
C(G+ij)-C(G)&\leq \frac{1}{\Omega_G(i,j)}+\sum_{pq\in E(G)}\frac{4\Omega^2_G(p,q)}{4\Omega^2_G(p,q)[1+\Omega_G(i,j)]-\Omega_G(p,q)\times4\Omega^2_G(p,q)}\\
&=\frac{1}{\Omega_G(i,j)}+\sum_{pq\in E(G)}\frac{1}{1+\Omega_G(i,j)-\Omega_G(p,q)}\\
&\leq \frac{1}{\Omega_G(i,j)}+\sum_{pq\in E(G)}\frac{1}{\Omega_G(i,j)} ~~~(\mbox{by} ~~1-\Omega_G(p,q)\geq 0)\\
&=\frac{m+1}{\Omega_G(i,j)}.
\end{align*}
Note that equality holds if and only if for any edge $pq\in E(G)$, $1-\Omega_G(p,q)=0$, $|\Omega_G(p,i)-\Omega_G(q,i)|=\Omega_G(p,q)$ and $|\Omega_G(p,j)-\Omega_G(q,j)|=\Omega_G(p,q)$. It indicates that $pq$ is an cut edge of $G$, either every path from $p$ to $i$ passes through $q$ or every path from $q$ to $i$ passes through $p$, and either every path from $p$ to $j$ passes through $q$ or every path from $q$ to $j$ passes through $p$. It is not hard to show that the above conditions are satisfied if and only if $G$ is a path of length $n-1$ and $i$ $\&$ $j$ are two end vertices of it.
\end{proof}

According to Theorem \ref{diff}, one could easily obtain the strictly monotone increasing property.

\begin{pro}\label{sm}(SMI-property)
Let $G$ be a connected graph with $n$ vertices and let $H$ be a connected
spanning subgraph of $G$ such that $G\neq H$. Then $$C(H)<C(G).$$
\end{pro}
\section{Some lower and upper bounds}
\subsection{Bounds via the SMI-property}
For convenience, we first compute the global cyclicity index of some special graphs. Let $T_n$, $C_n$, $K_n$ and $K_{n_1,n_2}$ denote a tree with $n$ vertices, the cycle graph with $n$ vertices, the complete graph with $n$ vertices, and the complete bipartite graph with partitions of size $n_1$ and $n_2$, respectively. It has been shown in \cite{ki} that $C(T_n)=0$ and $C(C_n)=\frac{n}{n-1}$. For $K_n$, since the resistance distance between every adjacent vertices is $\frac{2}{n}$, it is easily obtained that
\begin{equation}
C(K_n)=\frac{n(n-1)(n-2)}{4}.
\end{equation}
For $K_{n_1,n_2}$, as it has been shown that the resistance distance between adjacent vertices is $\frac{n_1+n_2-1}{n_1n_2}$, simple calculation leads to
\begin{equation}
C(K_{n_1,n_2})=\frac{n_1n_2(n_1n_2-n_1-n_2+1)}{n_1+n_2-1}.
\end{equation}

By the strictly monotone increasing property, we could establish bounds on the global cyclicity number of some classes of graphs.
\begin{thm}
For a connected graph $G$ with $n$ vertices, we have
\begin{equation}
0\leq C(G)\leq \frac{n(n-1)(n-2)}{4}.
\end{equation}
The first equality holds if and only if $G$ is a tree, and the second does if and only if $G=K_n$.
\end{thm}

\begin{thm}
For a connected bipartite graph $G$ with partitions of size $n_1$ and $n_2$, we have
\begin{equation}
0\leq C(G)\leq \frac{n_1n_2(n_1n_2-n_1-n_2+1)}{n_1+n_2-1}.
\end{equation}
The first equality holds if and only if $G$ is a tree, and the second does if and only if $G=K_{n_1,n_2}$.
\end{thm}

Now we turn to another interesting class of graphs--circulant graphs. A circulant graph is a graph of $n$ vertices in which the $i$th vertex is adjacent to the $(i+j)$th and $(i-j)$th vertices for each $j$ in a list $l$. Since any circulant graph of order $n$ has a hamilton cycle $C_n$ \cite{ma} and it is a subgraph of $K_n$, by Proposition \ref{sm}, we have

\begin{thm}
For a connected circulant graph $G$ with $n$ vertices, we have
\begin{equation}
\frac{n}{n-1}\leq C(G)\leq \frac{n(n-1)(n-2)}{4}.
\end{equation}
The first equality holds if and only if $G=C_n$, and the second does if and only if $G=K_n$.
\end{thm}

\subsection{Bound via an inequality for resistance distances}

We first give a lower bound for resistance distances between adjacent vertices. To obtain the lower bound, we will use a classical result in electrical network theory--Rayleigh's short-cut method \cite{ds}: Shorting certain sets of vertices together can only decrease the resistance distance of the network between two given vertices. Cutting certain edges
can only increase the resistance distance between two given vertices.

Let $G$ be a connected graph. For $v\in V(G)$, denote the degree of $v$ by $d_G(v)$. Then
\begin{lemma}\label{ouv}
The resistance distance between any two adjacent vertices $u,v\in V(G)$ satisfies:
\begin{equation}\label{e3}
\Omega_G(u,v)\geq \frac{d_G(u)+d_G(v)-2}{d_G(u)d_G(v)-1},
\end{equation}
with equality if and only if $d_G(u)=d_G(v)$ and there are $d_G(u)-1$ vertices that are adjacent both to $u$ and to $v$.
\end{lemma}
\begin{proof}
Denote by $G-uv$ the graph obtained from $G$ by deleting the edge $uv$. In $G-uv$, short all vertices of $V(G-uv)\setminus \{u,v\}$ together, obtaining a new graph $G_1$. Then by Rayleigh's short-cut method, $\Omega_{G-uv}(u,v)\geq \Omega_{G_1}(u,v)$. By serial and parallel connections, $\Omega_{G_1}(u,v)=\frac{1}{d_{G_1}(u)}+\frac{1}{d_{G_1}(v)}=\frac{1}{d_{G-uv}(u)}+\frac{1}{d_{G-uv}(v)}=\frac{1}{d_{G}(u)-1}+\frac{1}{d_{G}(v)-1}.$ Hence $\Omega_{G-uv}(u,v)\geq \frac{1}{d_{G}(u)-1}+\frac{1}{d_{G}(v)-1}$. Noticing that $\Omega_G(u,v)=\frac{\Omega_{G-uv}(u,v)}{1+\Omega_{G-uv}(u,v)}$ by the parallel connection, we obtain (\ref{e3}).

To have the equality in (\ref{e3}), it requires that $\Omega_{G-uv}(u,v)=\frac{1}{d_{G}(u)-1}+\frac{1}{d_{G}(v)-1}$. Now we show that if $\Omega_{G-uv}(u,v)=\frac{1}{d_{G}(u)-1}+\frac{1}{d_{G}(v)-1}$, then $d_{G}(u)=d_{G}(v)$. Suppose to the contrary that $d_{G}(u)\neq d_{G}(v)$. Without loss of generality, suppose that $d_{G}(u)<d_{G}(v)$. Choose a vertex which is adjacent to $v$ but not to $u$, say $w$. In $G-uv$, short all vertices of $V(G-uv)\setminus \{u,v,w\}$ together, obtaining a graph $G_2$. Then it is not difficult to observe that $\Omega_{G_2}(u,v)\geq \frac{2}{d_{G}(u)-1}$. Thus $\Omega_{G-uv}(u,v)\geq \Omega_{G_2}(u,v)\geq \frac{2}{d_{G}(u)-1}>\frac{1}{d_{G}(u)-1}+\frac{1}{d_{G}(v)-1}$, a contradiction. Hence the equality in (\ref{e3}) holds only if $d_{G}(u)=d_{G}(v)$. If $d_{G}(u)=d_{G}(v)$, then it is obvious that $\frac{d_G(u)+d_G(v)-2}{d_G(u)d_G(v)-1}=\frac{1}{d_G(u)+1}+\frac{1}{d_G(v)+1}$. Besides, it has been shown shown in \cite{lw} that $\Omega_G(u,v)=\frac{1}{d_G(u)+1}+\frac{1}{d_G(v)+1}$ if and only if $uv\in E(G)$, $d_G(u)=d_G(v)$ and there are $d_G(u)-1$ vertices that are adjacent both to $u$ and to $v$. Hence equality in (\ref{e3}) holds if and only if $d_G(u)=d_G(v)$ and there are $d_G(u)-1$ vertices that are adjacent both to $u$ and to $v$.
\end{proof}

By Lemma \ref{ouv}, we have
\begin{thm}\label{cd}
Let $G$ be a connected graph. Then
\begin{equation}\label{e4}
C(G)\leq \sum_{i\sim j}\big(\frac{1}{d_G(i)-1}+\frac{1}{d_G(j)-1}\big)^{-1},
\end{equation}
with equality if and only if $G$ is complete.
\end{thm}
\begin{proof}
By (\ref{e3}), for any adjacent vertices $i$ and $j$, we have
$$\sigma_G(i,j)-1\leq \frac{d_G(i)d_G(j)-1}{d_G(i)+d_G(j)-2}-1=\big(\frac{1}{d_G(i)-1}+\frac{1}{d_G(j)-1}\big)^{-1},$$
which yields (\ref{e4}). Suppose that $G$ satisfies the equality in (\ref{e4}). Then for any two adjacent vertices $k,l\in V(G)$, we have $d_G(k)=d_G(l)$. This indicates that $G$ is regular since $G$ is connected. Suppose that $G$ is $r$-regular. Then any two adjacent vertices of
$G$ have $r-1$ common neighbors. Choose any two adjacent vertices $i,j\in V(G)$ and suppose that the $r-1$ common neighbors of them are $i_1,i_2,\ldots,i_{r-1}$. For any vertex $i_s$ ($1\leq s\leq r-1$), since $i_s$ is adjacent to $i$, $i_s$ and $i$ have $r-1$ common neighbors and thus $i_1,\ldots,i_{s-1},i_{s+1},\ldots,i_{r-1},j$ must be their common neighbors. Hence the subgraph induced by $i,i_1,\ldots,i_{r-1},j$ is a complete graph of order $r+1$. But if $r+1\neq |V(G)|$, then the subgraph induced by $i,i_1,\ldots,i_{r-1},j$ would form a component of $G$ which contradicts the hypothesis that $G$ is connected. Hence $G$ is $(|V(G)|-1)$-regular, that is, $G$ is a complete graph.
\end{proof}

If $G$ is regular, then we could obtain an immediate consequence of Theorem \ref{cd}.
\begin{cor}
Let $G$ be a r-rugular graph with $n$ vertices. Then
\begin{equation}
C(G)\leq \frac{nr(r-1)}{4}.
\end{equation}
\end{cor}

Denote by $\Delta(G)$ the maximum degree of $G$. Then Theorem \ref{cd} yields
\begin{cor}
Let $G$ be a connected graph with $m$ edges. Then
\begin{equation}
C(G)\leq \frac{m(\Delta(G)-1)}{2}.
\end{equation}
\end{cor}

\subsection{Bounds via majorization techniques}
We first introduce some notions and results about the majorization order and Schur-convexity (see \cite{mo} for more details). Let
$$\mathcal{D}=\{\mathbf{x}\in\mathbb{R}^n: x_1\geq x_2\geq \ldots \geq x_n\}.$$
Given two vectors $\mathbf{y},\mathbf{z}\in \mathcal{D}$, the \textit{majorization order} $\mathbf{y}\unlhd \mathbf{z}$ means: $\sum\limits_{i=1}^n y_i=\sum\limits_{i=1}^n z_i$, and for $k=1,2,\ldots,n-1$, $\sum\limits_{i=1}^k y_i=\sum\limits_{i=1}^k z_i$. In this paper, we only consider some subsets of
$$\Sigma_\alpha=\mathcal{D}\cap\{\mathbf{x}\in \mathbb{R}_{+}^n:x_1+x_2+\ldots+x_n=\alpha\},$$
where $\alpha\in \mathbb{R}$, $\alpha>0$. Given a closed subset $S\in \Sigma_\alpha$, a vector $\mathbf{x}^*(S)\in S$ is said to be maximal for $S$ with respect to the majorization order if $\mathbf{x}\unlhd\mathbf{x}^*(S)$ for each $\mathbf{x}\in S$. Analogously, a vector $\mathbf{x}_*(S)\in S$ is said to be minimal for $S$ with respect to the majorization order if $\mathbf{x}_*(S)\unlhd\mathbf{x}$ for each $\mathbf{x}\in S$.

A symmetric function $\phi: A\rightarrow \mathbb{R}, A\subseteq \mathbb{R}^n$, is said to be \textit{Schur-convex} on $A$ if $\mathbf{x}\unlhd \mathbf{y}$ implies $\phi(\mathbf{x})\leq \phi(\mathbf{y})$. If in addition $\phi(\mathbf{x})<\phi(\mathbf{y})$ for $\mathbf{x}\unlhd \mathbf{y}$ but $\mathbf{x}$ is not a permutation of $\mathbf{y}$, $\phi$ is said to be strictly Schur-convex on $A$.

\begin{pro}\cite{mo}
Let $I\subset \mathbb{R}$ be an interval and let $\phi(x)=\sum\limits_{i=1}^ng(x_i)$, where $g: I\rightarrow \mathbb{R}$. If $g$ is strictly convex on $I$, then $\phi$ is strictly Schur-Convex on $I^n=\underbrace{I\times\cdots \times I}_{n-times}$.
\end{pro}

\begin{lemma}\cite{mo}\label{lub}
Let $0\leq m<M$ and $m\leq\frac{\alpha}{n}\leq M$. Given the subset
$$S=\Sigma_\alpha=\mathcal{D}\cap\{\mathbf{x}\in \mathbb{R}^n:M\geq x_1\geq x_2\geq \ldots \geq x_n\geq m\}$$ we have
$$\mathbf{x}_*(S)=(\frac{\alpha}{n},\frac{\alpha}{n},\ldots,\frac{\alpha}{n}),$$
and
$$\mathbf{x}^*(S)=(\underbrace{M,\ldots,M}_k,\theta,\underbrace{m,\ldots,m}_{n-k-1}),$$
where $k=\displaystyle\lfloor\frac{\alpha-nm}{M-m}\rfloor$ and $\theta=\alpha-Mk-m(n-k-1)$.
\end{lemma}

We also need to show two fundamental facts on resistance distances as given in the following two lemmas.

\begin{lemma}\label{day}
Let $G$ be a connected graph with $n$ vertices ($n\geq 3$). If $G$ has a cut edge, then for any two vertices $u,v\in V(G)$,
$$\Omega_G(u,v)>\frac{2}{n}.$$
\end{lemma}
\begin{proof}
Suppose that $e=ij$ is a cut edge of $G$, and let $G_1$, $G_2$ be two components of $G-e$ such that $i\in V(G_1)$, $j\in V(G_2)$. If $u$ and $v$ are in the same component, say $G_1$, then $\Omega_G(u,v)=\Omega_{G_1}(u,v)\geq \frac{2}{|V(G_1)|}>\frac{2}{n}$. If $u$ and $v$ are in different components, say $u\in G_1$ and $v\in G_2$, then $\Omega_G(u,v)=\Omega_G(u,i)+\Omega_G(i,j)+\Omega_G(j,v)>\Omega_G(i,j)=1>\frac{2}{n}$.
\end{proof}
\begin{lemma}\label{2n}
Let $G$ be a connected graph of order $n$ ($n\geq 3$). If $G$ is not complete, there are at most $\frac{n^2-5n+8}{2}$ pairs of vertices with resistance distances being equal to $\frac{2}{n}$.
\end{lemma}
\begin{proof}
Denote by $K_n^-$ the graph obtained from $K_n$ by deletion of an edge $uv$. We first show that there are exactly $\frac{n^2-5n+6}{2}$ pairs of vertices in $K_n^-$ with resistance distances being equal to $\frac{2}{n}$.  It is not hard to compute that $\Omega_{K_n^-}(u,v)=\frac{2}{n}$ and for any pair of vertices $\{x,y\}\neq \{u,v\} $, $\Omega_{K_n^-}(x,y)=\frac{2}{n}$. But for any vertex $w\in V(K_n^-)\setminus\{u,v\}$, by Lemma \ref{ouv}, we have
$$\Omega_{K_n^-}(w,u)=\Omega_{K_n^-}(w,v)\geq \frac{n-1+n-2-2}{(n-1)(n-2)-1}=\frac{2n-5}{n^2-3n+1}>\frac{2}{n}.$$
Hence in $K_n^-$, there are exactly $\frac{n(n-1)}{2}-2(n-2)=\frac{n^2-5n+8}{2}$ pairs of vertices with resistance distances being equal to $\frac{2}{n}$. Note that $G$ is a subgraph of $K_n^-$ since $G$ is not complete. Then the desired result follows since $\Omega_{G}(u,v)\geq \Omega_{K_n^-}(u,v)$ for all $u,v\in V(G)$ according to Rayleigh short-cut method.
\end{proof}

Before proving the main results of this section, we still need to define a class of graphs which will be frequently used in what follows.
\begin{defn}
A connected graph $G$ is called \textit{electrically-edge-equivalent} if all the resistance distances between pairs of adjacent vertices are equal, that is, for any two edges $e_1=uv$ and $e_2=xy$, $\Omega_G(u,v)=\Omega_G(x,y)$.
\end{defn}
As can be easily seen, electrically-edge-equivalent graphs contain some interesting classes of graphs, such as trees and edge-transitive graphs.

Now we are ready to give our main results in this section.

\begin{thm}\label{bound}
Let $G$ be a connected graph with $n$ vertices ($n\geq 3$) and $m$ edges. Then
\begin{equation}\label{ie1}
\frac{m(m-n+1)}{n-1}\leq C(G)\leq \frac{n}{(n-2)\epsilon+2}+\frac{mn-n^2+(n-2)\epsilon}{2},
\end{equation}
where $\epsilon=\frac{n^2-n-2m}{n-2}-\lfloor\frac{n^2-n-2m}{n-2}\rfloor$.
Moreover, the first equality holds if and only if $G$ is electrically-edge-equivalent, and the second does if and only if $G$ is a tree or $G=K_n$.
\end{thm}
\begin{proof}
Let $e_1$, $e_2$,\ldots,$e_m$ be edges of $G$. For simplicity, denote the resistance distance between the end vertices of $e_i$ by $\Omega_i$, $1\leq i\leq m$. Without loss of generality, suppose that $\Omega_1\geq \Omega_2\geq\ldots\geq \Omega_m$. Again by the Foster's first formula mentioned in the second section, we have $\Omega_1+\Omega_2+\ldots+ \Omega_m=n-1.$
In order to obtain the desired inequalities, we evaluate the extremal values of the Schur-convex function
\begin{equation}
 f(\Omega_1,\Omega_2,\ldots,\Omega_m)=\sum_{i=1}^{m}\frac{1}{\Omega_i},
\end{equation}
and consider the set
$$S=\{\Omega\in \mathbb{R}^{m}:\sum_{i=1}^{m}\Omega_i=n-1, \frac{2}{n}\leq\Omega_m\leq \Omega_{m-1}\leq\ldots\leq \Omega_1\leq 1 \}.$$
It is easily verified that $f(\Omega_1,\Omega_2,\ldots,\Omega_m)$ is strictly Schur-convex on $S$.
By Lemma \ref{lub}, we know that the minimal element of $S$ with respect to majorization order is given by
$$(\underbrace{\frac{n-1}{m},\frac{n-1}{m},\ldots,\frac{n-1}{m}}_m).$$
The function $f$ attains its minimum at this point with the minimum value given by $\frac{m^2}{n-1}$.
Hence the minimum value of $C(G)$ is $$\frac{m^2}{n-1}-m=\frac{m(m-n+1)}{n-1}.$$ Note that the minimum value of $C(G)$ is achieved if and only if resistance distances between all pairs of adjacent vertices are equal. Hence the first equality in (\ref{ie1}) holds if and only if $G$ is electrically-edge-equivalent.

Now we prove the upper bound. By Lemma \ref{lub}, the maximal element of $S$ with respect to the majorization order is given by
$$(\underbrace{1,\ldots,1}_{k},\theta,\underbrace{\frac{2}{n},\ldots,\frac{2}{n}}_{m-k-1}),$$
where $k=\lfloor\frac{n-1-m\frac{2}{n}}{1-\frac{2}{n}}\rfloor=\lfloor\frac{n^2-n-2m}{n-2}\rfloor$ and $\theta=n-1-k-\frac{2}{n}(m-k-1)$. The function $f$ attains its minimum at this point with the minimum value given by $k+\frac{1}{\theta}+\frac{n(m-k-1)}{2}$.

Hence if we let $\epsilon=\frac{n^2-n-2m}{n-2}-k$, then
\begin{align*}
&C(G)\leq k+\frac{1}{\theta}+\frac{n(m-k-1)}{2}-m\\
&=k+\frac{1}{n-1-k-\frac{2}{n}(m-k-1)}+\frac{n(m-k-1)}{2}-m\\
&=k+\frac{n}{n^2-n-2m-(n-2)k+2}+\frac{mn}{2}-\frac{nk}{2}-\frac{n}{2}-m\\
&=\frac{n}{n^2-n-2m-(n-2)k+2}-\frac{(n-2)k}{2}+\frac{mn}{2}-\frac{n}{2}-m\\
&=\frac{n}{n^2-n-2m-(n-2)(\frac{n^2-n-2m}{n-2}-\epsilon)+2}-\frac{(n-2)(\frac{n^2-n-2m}{n-2}-\epsilon)}{2}+\frac{mn}{2}-\frac{n}{2}-m\\
&=\frac{n}{(n-2)\epsilon+2}-\frac{n^2-n-2m-(n-2)\epsilon}{2}+\frac{mn}{2}-\frac{n}{2}-m\\
&=\frac{n}{(n-2)\epsilon+2}+\frac{mn-n^2+(n-2)\epsilon}{2}.
\end{align*}

Note that achieving the second equality in (\ref{ie1}) requirs $\Omega_1=\Omega_2=\ldots=\Omega_k=1$, $\Omega_{k+1}=\Theta$ and $\Omega_{k+2}=\Omega_{k+3}=\ldots=\Omega_m=\frac{2}{n}$. Actually, as shown below, the requirement can be satisfied only when $G$ is a tree or $G=K_n$.
For convenience, we distinguish the following three cases.

\textit{Case 1.} $k=m-1$. In this case, $\theta=n-1-(m-1)=n-m$. Together with the fact $1\geq\theta\geq \frac{2}{n}>0$, we have $\theta=1$. Hence $m=n-1$ and $G$ is a tree.

\textit{Case 2.} $1\leq k\leq m-2$. By $\Omega_1=\Omega_2=\ldots=\Omega_k=1$, we know that $e_1,e_2,\ldots,e_k$ are cut edges of $G$. By Lemma \ref{day}, for any two vertices $i,j\in V(G)$, $\Omega_{G}(i,j)>\frac{2}{n}$. But this clearly contradicts the fact that $\Omega_{k+2}=\Omega_{k+3}=\ldots=\Omega_m=\frac{2}{n}$. Hence in this case, there does not exist any graph satisfying the desired condition.

\textit{Case 3.} $k=0$. In this case, $\Omega_1=\theta=n-1-\frac{2}{n}(m-1)$ and $\Omega_{2}=\Omega_{3}=\ldots=\Omega_m=\frac{2}{n}$. Now we show that $m=\frac{n(n-1)}{2}$. Suppose to the contrary that $m<\frac{n(n-1)}{2}$. Then $G$ is not complete. So on one hand, by Lemma \ref{2n}, there are at most $\frac{n^2-5n+8}{2}$ pairs of vertices with resistance distances being equal to $\frac{2}{n}$; on the other hand, since it is readily verified that $m> \frac{n(n-2)}{2}+1$ according to $\Omega_1=\theta=n-1-\frac{2}{n}(m-1)<1$, there are $m-1>\frac{n^2-2n}{2}$ pairs of vertices in $G$ with resistance distances being equal to $\frac{2}{n}$. But this is an obvious contradiction since $\frac{n^2-2n}{2}> \frac{n^2-5n+8}{2}$. Thus $m=\frac{n(n-1)}{2}$ and $G$ is complete.
\end{proof}

As the upper bound in Theorem \ref{bound} is somewhat complicated, in the following we give a much simpler upper bound by throwing $\epsilon$ off.

\begin{cor}\label{bouud1}
Let $G$ be a connected graph with $n$ vertices ($n\geq 3$) and $m$ edges. Then
\begin{equation}\label{e6}
C(G)\leq \frac{n(m-n+1)}{2},
\end{equation}
with equality if and only if $G$ is a tree or $G=K_n$.
\end{cor}
\begin{proof}
By Theorem \ref{bound}, we have
$$C(G)\leq \frac{n}{(n-2)\epsilon+2}+\frac{mn-n^2+(n-2)\epsilon}{2},$$
where $\epsilon=\frac{n^2-n-2m}{n-2}-\lfloor\frac{n^2-n-2m}{n-2}\rfloor$. Clearly, $\epsilon\in [0,1)$. Now define the function $f$ of $x$ as
$$f(x)=\frac{n}{(n-2)x+2}+\frac{mn-n^2+(n-2)x}{2}.$$
In what follows, we will show that for real variable $x\in [0,1)$, $f(x)$ attains its maximum value at $x=0$,
which thus will lead to $C(G)\leq \frac{n}{2}+\frac{mn-n^2}{2}=\frac{n(m-n+1)}{2}$ as desired. To this end, consider the first derivative of $f$ with respect to $x$,
$$f'(x)=\frac{n(n-2)}{-[(n-2)x+2]^2}+\frac{n-2}{2}.$$
Solving $f'(x)=0$, we have $x=\frac{\sqrt{2n}-2}{n-2}.$ The only thing left is to compare $f(0)$ with $f(\frac{\sqrt{2n}-2}{n-2})$.
\begin{align*}
f(0)-f(\frac{\sqrt{2n}-2}{n-2})&=\frac{n}{2}+\frac{mn-n^2}{2}-(\frac{n}{\sqrt{2n}-2+2}+\frac{mn-n^2+\sqrt{2n}-2}{2})\\
&=\frac{n}{2}-\sqrt{2n}+1>0.
\end{align*}
Hence $f(x)$ attains its maximum value at $x=0$.

In view of the fact that the value of $\epsilon$ corresponding to $G$ is 0 if $G$ is a tree or $G=K_n$, we could conclude that the equality in (\ref{e6}) holds if and only if $G$ is a tree or $G=K_n$.
\end{proof}

Noticing that $m-n+1$ is none other than the cyclomatic number $\mu(G)$, Theorem \ref{bound} and Corollay \ref{bouud1} yield a nice relationship between $C(G)$ and $\mu(G)$.
\begin{thm}\label{cu}
For a connected graph $G$ with $n$ vertices and $m$ edges, we have
\begin{equation}
\frac{m}{n-1}\mu(G)\leq C(G)\leq \frac{n}{2}\mu(G).
\end{equation}
The first equality holds if and only if $G$ is electrically-edge-equivalent, and the second does if and only if $G$ is a tree or $G=K_n$.
\end{thm}

\section{Nordhaus-gaddum-type results}
A Nordhaus-Gaddum-type result is a (tight) lower or upper bound on the sum or product of a parameter of a graph and its complement \cite{ng}. In this section, we will consider Nordhaus-Gaddum-type results for the global cyclicity index. There is only one connected
graph $P_4$ on 4 vertices with the connected complement $\overline{P_4}=P_4$. Since $C(P_4)=0$, we have $C(P_4)+C(\overline{P_4})=C(P_4)C(\overline{P_4})=0$.
\begin{thm}\label{ng}
Let $G$ be a connected (molecular) graph on $n\geq 5$ vertices with
a connected $\overline{G}$. Then
\begin{equation}
\frac{n(n-1)(n-4)}{8}\leq C(G)+C(\overline{G})< \frac{n(n-1)(n-4)}{4},
\end{equation}
with equality if and only if $|E(G)|=|E(\overline{G})|$ and both $G$ and $\overline{G}$ are electrically-edge-equivalent.
\end{thm}
\begin{proof}
Suppose that $|E(G)|=m$. Then $|E(\overline{G})|=\frac{n(n-1)}{2}-m$. According to the lower bound in Theorem \ref{bound}, we have
\begin{align*}
C(G)+C(\overline{G})&\geq \frac{m^2}{n-1}-m+\frac{[\frac{n(n-1)}{2}-m]^2}{n-1}-[\frac{n(n-1)}{2}-m]\\
&=\frac{m^2+[\frac{n(n-1)}{2}-m]^2}{n-1}-\frac{n(n-1)}{2}\\
&\geq \frac{2[\frac{n(n-1)}{4}]^2}{n-1}-\frac{n(n-1)}{2}=\frac{n(n-1)(n-4)}{8}.
\end{align*}
Equality holds if and only if $C(G)=\frac{m^2}{n-1}-m$, $C(\overline{G})=\frac{[\frac{n(n-1)}{2}-m]^2}{n-1}-[\frac{n(n-1)}{2}-m]$, and $m=\frac{n(n-1)}{2}-m$, that is, $|E(G)|=|E(\overline{G})|$ and both $G$ and $\overline{G}$ are electrically-edge-equivalent.

For the upper bound, since neither $G$ nor $\overline{G}$ is complete and it is impossible that both $G$ and $\overline{G}$ are trees, by Corollary \ref{bouud1}, we have
\begin{align*}
C(G)+C(\overline{G})&< \frac{n(m-n+1)}{2}+\frac{n(\frac{n(n-1)}{2}-m-n+1)}{2}\\
&=\frac{n(\frac{n(n-1)}{2}-2n+2)}{2}=\frac{n(n-1)(n-4)}{4}.
\end{align*}
\end{proof}

It is interesting to note for any two connected graphs $G_1$ and $G_2$ of the same order, $C(G_1)+C(\overline{G_1})$ is no more than twice of $C(G_2)+C(\overline{G_2})$ and no less than half of $C(G_2)+C(\overline{G_2})$, though $C(G_1)$ and $C(G_2)$ may different very much.

\textit{Remark 1.} Graphs satisfying $|E(G)|=|E(\overline{G})|$ and both $G$ and $\overline{G}$ are electrically-edge-equivalent do exist. An typical example is the famous Paley graph. For a prime power $q\equiv 1$ (mod 4), a Paley graph $\mathrm{P}_q$ is the graph with vertices the elements of the finite field $\mathbf{F}_q$ and an edge between $x$ and $y$ if and only if
$x-y$ is a non-zero square in $\mathbf{F}_q$. Paley graphs satisfy the desired conditions since they are self-complementary and edge-transitive \cite{bo}.

\textit{Remark 2.} Although the upper bound in Theorem \ref{ng} is not tight, it is asymptotically tight. To see this, let us consider graphs $G$ and $\overline{G}$ as shown in Figure 1. It is readily seen that the subgraph induced by vertices $3,4,\ldots,n$ is the complete graph $K_{n-2}$. Hence $C(\overline{G})$ is equal to the global cyclicity number of the complete graph $K_{n-2}$, that is,
$$C(\overline{G})=\frac{(n-2)(n-3)(n-4)}{4}.$$ Thus
$$C(G)+C(\overline{G})>\frac{(n-2)(n-3)(n-4)}{4}.$$
Hence we could conclude that the upper bound is asymptotically tight.

\begin{figure}[h]
\begin{center}
\includegraphics[scale=0.60
]{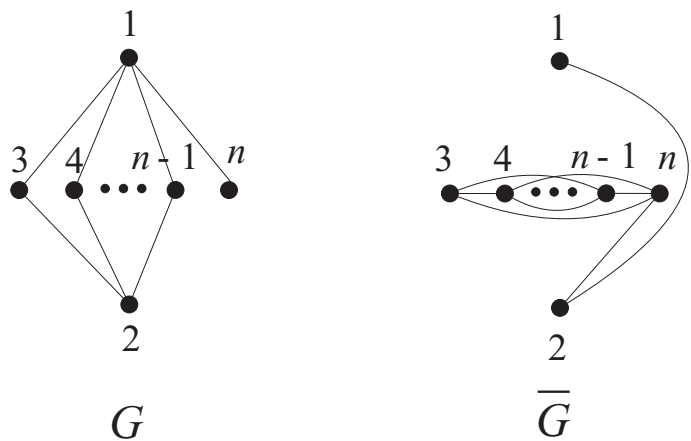}\caption{Illustration of graphs $G$ and $\overline{G}$ in Remark 2.}
\end{center}
\end{figure}

We complete this section by giving lower and upper bounds for the product of $C(G)$ and $C(\overline{G})$.
\begin{thm}
Let $G$ be a connected (molecular) graph on $n\geq 5$ vertices with
a connected $\overline{G}$. Then
\begin{equation}
0\leq C(G)C(\overline{G})<\big(\frac{n(n-1)(n-4)}{8}\big)^2,
\end{equation}
with equality if and only if $G$ or $\overline{G}$ is a tree.
\end{thm}
\begin{proof}
Since the global cyclicity index of a tree is 0, the lower bound is trivial. For the upper bound, suppose that $|E(G)|=m$. Then by Corollary \ref{bouud1}, we have
\begin{align*}
C(G)C(\overline{G})&< \frac{n(m-n+1)}{2}\times\frac{n(\frac{n(n-1)}{2}-m-n+1)}{2}\\
&=\frac{n^2}{4}\big[\frac{n(n-1)}{2}m-m^2-\frac{n(n-1)^2}{2}+(n-1)^2\big]\\
&=\frac{n^2}{4}\big[-\big(m-\frac{n(n-1)}{4}\big)^2+\big(\frac{n(n-1)}{4}\big)^2-\frac{n(n-1)^2}{2}+(n-1)^2\big]\\
&\leq \frac{n^2}{4}\big[\big(\frac{n(n-1)}{4}\big)^2-\frac{n(n-1)^2}{2}+(n-1)^2\big]\\
&=\big(\frac{n(n-1)(n-4)}{8}\big)^2.
\end{align*}
\end{proof}

\section{Acknowledgement}
The author acknowledges the support of National Natural Science Foundation of China under Grant No.11201404, China Postdoctoral Science Foundation under Grant No.2012M521318, Special Funds for Postdoctoral Innovative Projects of Shandong Province under Grant No.201203056.

\end{document}